\documentclass[11pt]{article}
\usepackage[T1]{fontenc}
\usepackage{lmodern,amsmath,amssymb,amsthm,mathtools}
\usepackage[margin=1in]{geometry}
\usepackage{booktabs,array,longtable,microtype,enumitem,xurl,needspace}
\usepackage[colorlinks=true,linkcolor=blue!45!black,citecolor=blue!45!black,urlcolor=blue!45!black]{hyperref}
\usepackage{xcolor}
\hypersetup{pdftitle={Tiling without a subspace complement: a sharp counterexample in dimension nine over the field with two elements},pdfauthor={Samuel J. Ferguson and Joshua Iosevich},pdfsubject={Spectral tiles without subspace complements and the sharp dimension threshold}}
\setlist{itemsep=3pt,topsep=4pt}
\numberwithin{equation}{section}
\newtheorem{theorem}{Theorem}[section]
\newtheorem{lemma}[theorem]{Lemma}
\newtheorem{proposition}[theorem]{Proposition}
\newtheorem{corollary}[theorem]{Corollary}
\theoremstyle{definition}
\newtheorem{definition}[theorem]{Definition}

\newtheorem{question}[theorem]{Question}
\theoremstyle{remark}
\newcommand{\F}{\mathbb F_2}
\newcommand{\wt}{\operatorname{wt}}
\newcommand{\supp}{\operatorname{supp}}
\newcommand{\spanof}{\operatorname{span}}
\newcommand{\W}{\mathcal W}
\newcommand{\wsp}{w_{\mathrm{sp}}}
\newcommand{\one}{\mathbf1}
\title{\Large\bfseries Tiling without a subspace complement: a sharp counterexample\\
in dimension nine over the field with two elements}
\author{Samuel J. Ferguson \and Joshua Iosevich}
\date{29 October 2026}
\begin{document}
\maketitle
\begin{abstract}
We construct an explicit 64-point subset of $\mathbb F_2^9$ which tiles by translation but has no linear-subspace tiling complement. Together with our earlier dimension-eight theorem and a product argument, this shows that every nonempty tile in $\mathbb F_2^d$ has a linear-subspace complement precisely when $d\le8$. The example is spectral, so it is not a counterexample to Fuglede's conjecture. Its construction uses six three-dimensional subspaces of $\mathbb F_2^6$, specified by a small table. We classify all its tiling complements and spectra: there are exactly 512 complements and 64 spectra, with one translation class of each. Every spectrum has normalized Wiener norm $15/8$. An appendix proves that Goyal's related construction is affinely inequivalent but has the same tiling complements and spectra. The construction and both classifications are proved from first principles. We also give elementary proofs of the positive assertion in dimensions at most seven and for all tile sizes other than sixteen in dimension eight. The remaining case uses our previously established theorem based on the classification of perfect binary codes of length fifteen.
\end{abstract}
\setcounter{tocdepth}{1}{\small\tableofcontents}
\clearpage

\section{Introduction}
A subset of a finite group tiles if its translates partition the group. A particularly simple way for this to happen in a vector space is to choose one representative from each coset of a subspace. The subspace is then a tiling complement. Must every tile arise in this way, perhaps using a different complement from the one first supplied?

This paper gives a concrete negative answer in dimension nine over the field with two elements. We exhibit a set which has many tiling complements, but none is a subspace. More strongly, every complement is a translate of one fixed eight-point configuration. We also determine every spectrum of the set, where a spectrum is an orthogonal basis of sign-valued functions obtained from linear functionals. These notions will be defined before they enter the proofs.

The example separates two questions. The finite version of Fuglede's conjecture asks whether tiling and spectrality are equivalent. It originates in Fuglede's 1974 work on orthogonal exponential functions and translation tilings~\cite{Fuglede}. Our example has \emph{both} properties. It disproves the stronger requirement that a tile, or a spectral set, must admit a linear-subspace complement. It therefore does not settle the spectral-to-tiling implication in dimension nine.

The precise dimension statement is the following.
\begin{theorem}[Sharp dimension threshold]\label{thm:threshold}
For a positive integer $d$, every nonempty translational tile in $\F^d$ has a linear-subspace tiling complement if and only if $d\le8$.
\end{theorem}

The negative direction of Theorem~\ref{thm:threshold} follows from the construction and a product argument. For the positive direction, our earlier preprint~\cite[Theorem~1.1]{FI} proves that tiling, spectrality, and the existence of a linear-subspace complement are equivalent in dimensions seven and eight. Here we give elementary proofs of the tiling implication in all dimensions at most seven and for every tile size other than sixteen in dimension eight. The remaining sixteen-point case is the sole external input to our sharp-threshold proof: Theorem~\ref{thm:external16}, taken from Ferguson--Iosevich~\cite[Lemma~6.5]{FI}, depends on the classification of perfect binary codes of length fifteen. This input completes the low-dimensional positive statement; the dimension-nine construction, its tiling and spectrality, the absence of a subspace complement, and both complete classifications are proved within this paper without it. Apart from that explicitly cited input, the sharp-threshold proof is elementary and self-contained relative to the prerequisites below. Understanding the statement and role of the input does not require reconstructing the classification of perfect codes.

We assume modular arithmetic; finite abelian groups written additively, including subgroups and cosets; and finite-dimensional linear algebra over fields, including vector spaces, linear combinations, bases, dimension, linear maps, kernels, rank--nullity, matrix multiplication, transposes, inverses, and real or complex orthogonality. The counting arguments use finite sums, bijections, parity, and binomial coefficients, together with elementary proof techniques such as contradiction. Relevant algebraic facts are recalled before use. No prior knowledge of Fourier analysis, spectral sets, Fuglede's conjecture, coding theory, or graph theory is required for the construction and classification arguments. We introduce the necessary definitions and prove the Fourier and combinatorial facts as they arise; this is the sense in which those arguments proceed from first principles.

The eight-point configuration underlying the example, and the fact that its pair-sums contain no three-dimensional subspace, come from our earlier example~\cite[Proposition~7.1]{FI}. In that example the domain still has a subspace complement. Our construction combines several subspaces in different layers so that \emph{every} tiling complement is forced into that configuration. The edge argument used there also gives the complete classification here.

For context, Ferguson and Sothanaphan~\cite{FS} studied the finite-field conjecture and established, among other results, a counterexample in $\F^{10}$. Kadowaki's preprint~\cite{Kadowaki} establishes the dimension-seven equivalence and introduces a common-spectrum construction for complements of small sets. Our broader preliminary preprint~\cite{FIpreliminary} contains further dimension-nine reductions. None of these general results is needed to verify the explicit example below. We make no priority claim for the construction or its classification.

Goyal~\cite{Goyal} gives a related spectral tile without a periodic
complement, classifies its tiling complements, and supplies a Lean
formalization. Appendix~\ref{app:comparison} compares the two explicit
domains. They are affinely inequivalent, although their complete families
of tiling complements and spectra coincide in the displayed coordinates.
We prove the spectrum classification for a common class of layer
constructions, which includes both examples. The periodic tiling
consequence for our domain is recorded after its spectral analysis.

We introduce the set in Section~\ref{sec:construction}, establish its tiling and exact difference set in Section~\ref{sec:differences}, and classify its complements in Section~\ref{sec:complements}. Only then do we develop the Fourier facts needed to classify its spectra. The low-dimensional positive results appear after the counterexample has been completely analyzed. A final product argument proves failure in every larger dimension.

\section{The set: six small spaces in eight layers}\label{sec:construction}
\subsection{Vectors over the field with two elements and tiling}
We write $\F=\mathbb Z/2\mathbb Z=\{0,1\}$. Coordinates are added modulo two, so $x+x=0$ and $x-y=x+y$. For example,
\[
 (1,0,1)+(0,1,1)=(1,1,0).
\]
A vector space of dimension $r$ over $\F$ has $2^r$ elements: each vector is a unique sum of a subset of the $r$ vectors in a basis. A \emph{linear subspace} contains zero and is closed under addition. An \emph{affine subspace} is a translate $a+K=\{a+k:k\in K\}$ of a linear subspace. For instance, in $\F^2$, $\{00,10\}$ is a linear subspace and $\{01,11\}$ is an affine subspace not containing zero. Such a translate is also called a \emph{coset} of $K$.

For subsets $A,B$ of a vector space over $\F$, their \emph{sumset} is
\[
 A+B=\{a+b:a\in A,\ b\in B\}.
\]
It records the resulting vectors, without their multiplicities. Thus $A+A$ is also the difference set of $A$. For $A=\{00,10\}$ we have $A+A=A$.

\begin{definition}
Nonempty sets $A,T\subseteq G=\F^d$ \emph{tile by translation}, written $A\oplus T=G$, if each $g\in G$ has exactly one representation $g=a+t$ with $a\in A$ and $t\in T$. We call $T$ a \emph{tiling complement} of $A$.
\end{definition}
The symbol $\oplus$ here asserts unique representation; neither set need be a subspace. For example,
\[
 \{00,10\}\oplus\{00,01\}=\F^2.
\]
Counting the pairs immediately gives $|A||T|=2^d$.

The following criterion is the basic way we will check tiling. A collision between two representations produces a common difference.
\begin{lemma}[Disjoint differences]\label{lem:tiling}
For nonempty $A,T\subseteq\F^d$,
\[
 A\oplus T=\F^d
 \quad\Longleftrightarrow\quad
 |A||T|=2^d\ \text{ and }\ (A+A)\cap(T+T)=\{0\}.
\]
In particular, if $|A|=2^{d-k}$ and $K$ is a $k$-dimensional subspace, then $A\oplus K=\F^d$ exactly when $(A+A)\cap K=\{0\}$.
\end{lemma}
\begin{proof}
An equality $a+t=a'+t'$ is equivalent to $a+a'=t+t'$. If the common vector must be zero, then $a=a'$ and $t=t'$, so all $|A||T|$ sums are distinct. When that number is $2^d$, they exhaust the group. Conversely, unique representation forbids any nonzero common difference, and gives the stated count. Finally, a subspace satisfies $K+K=K$ and $|K|=2^k$.
\end{proof}

\subsection{The eight-point configuration}
Put $H=\F^6$, let $e_1,\ldots,e_6$ be its coordinate basis, and put $j=e_1+\cdots+e_6$. Define
\begin{equation}\label{eq:C}
 C=\{0,e_1,e_2,e_3,e_4,e_5,e_6,j\},\qquad S=C+C.
\end{equation}
The \emph{Hamming weight} $\wt(x)$ of a vector over $\F$ is the number of its coordinates equal to one. For example, $\wt(1,0,1)=2$. Direct addition of the points of $C$ gives
\begin{equation}\label{eq:SZ}
 S=\{x\in H:\wt(x)\in\{0,1,2,5,6\}\},\qquad
 Z=H\setminus S=\{x:\wt(x)\in\{3,4\}\}.
\end{equation}
Indeed, its nonzero pair-sums are $e_i$, $e_i+e_k$ for $i\ne k$, $j+e_i$, and $j$. The number of vectors of weight $k$ is $\binom6k$, because one chooses which $k$ coordinates are one. Hence
\[
 |S|=1+6+15+6+1=29,\qquad |Z|=20+15=35.
\]

We use short integer labels for vectors in the next table. Precisely, the integer $\sum_{i=1}^6x_i2^{i-1}$ denotes $(x_1,\ldots,x_6)$. Thus $7=e_1+e_2+e_3$, $28=e_3+e_4+e_5$, and $63=j$. Addition of these labels means coordinatewise addition modulo two, \emph{not ordinary integer addition}: for example $7+28=27$ in this notation.

\subsection{A complete description of the domain}
Table~\ref{tab:spaces} specifies six subspaces of $H$. In each row the first three entries are a basis $a,b,c$; the next four entries give all the remaining nonzero vectors. Thus the table contains every vector needed to check the construction by hand.
\begin{table}[ht]
\centering
\begin{tabular}{c|rrr|rrr|r}
\toprule
Space & $a$&$b$&$c$&$a+b$&$a+c$&$b+c$&$a+b+c$\\\midrule
$A_0$&7&28&42&27&45&54&49\\
$A_1$&11&28&37&23&46&57&50\\
$A_2$&21&26&35&15&54&57&44\\
$A_3$&22&25&37&15&51&60&42\\
$D_0$&14&19&41&29&39&58&52\\
$D_1$&13&19&38&30&43&53&56\\\bottomrule
\end{tabular}
\caption{Six three-dimensional subspaces, using the integer labels defined above. Add zero to each row to obtain the whole subspace.}\label{tab:spaces}
\end{table}
Each row consists of seven distinct nonzero vectors, so its three proposed basis vectors are independent. Indeed, a dependence would make one of their seven nonempty sums zero. The first three and the last entries in every row have weight three; the middle three entries have weight four.

Let $Y=\F^3$, whose vectors are labeled $0,\ldots,7$ in the same way. In $G=H\times Y=\F^9$, define
\begin{equation}\label{eq:layers}
 (B_0,B_1,B_2,B_3,B_4,B_5,B_6,B_7)
 =(A_0,A_1,A_2,D_0,A_3,D_1,A_0,D_0)
\end{equation}
and
\begin{equation}\label{eq:E}
 \boxed{\displaystyle E=\bigcup_{y\in Y}(B_y\times\{y\}).}
\end{equation}
There are eight disjoint layers, with eight points in each, so $|E|=64$. For integer encoding in nine dimensions, $(x,y)$ is the integer $x+64y$, where this particular expression is ordinary integer arithmetic.

Here are the conclusions to be proved before we discuss lower dimensions.
\Needspace{12\baselineskip}
\begin{theorem}[The explicit example]\label{thm:example}
For the set $E$ in \eqref{eq:E}, the tiling complements are exactly
\[
 (C\times\{0\})+g\qquad(g\in G).
\]
There are 512 distinct such complements, none a linear or affine subspace. The spectra of $E$ are exactly
\[
 (C+a)\times Y\qquad(a\in H).
\]
There are 64 distinct spectra. In each classification there is one translation class and precisely eight sets containing zero. In particular, $E$ is both a tile and a spectral set, but it has no linear-subspace tiling complement.
\end{theorem}
The term spectrum is defined in Section~\ref{sec:fourier}; it plays no role in verifying the tiling or excluding subspace complements.

\section{Why the set tiles, and its exact differences}\label{sec:differences}
The construction has two features. Within each layer, the eight points of $C$ provide a complement. Between suitable different layers, every horizontal difference occurs. This leaves very few possible differences for any tiling complement of $E$.

\begin{lemma}\label{lem:table}
The six spaces in Table~\ref{tab:spaces} satisfy:
\begin{enumerate}
\item each has no nonzero point in $S$;
\item their union is $Z\cup\{0\}$;
\item $A_i\cap D_k=\{0\}$ for $0\le i\le3$ and $0\le k\le1$.
\end{enumerate}
\end{lemma}
\begin{proof}
All their nonzero vectors have weight three or four, proving the first assertion. Among the 42 nonzero entries of the table, exactly the seven numbers
\[
 15,19,28,37,42,54,57
\]
appear twice; every other entry appears once. This gives $42-7=35$ distinct vectors, all in the 35-element set $Z$, proving the second assertion. Finally, none of the entries in either $D$ row occurs in any $A$ row, proving the third assertion. These checks require only the displayed table and the binary expansions of its entries.
\end{proof}

\begin{proposition}\label{prop:tile}
The set $T_0=C\times\{0\}$ is a tiling complement of $E$.
\end{proposition}
\begin{proof}
For each $y$, $B_y+B_y=B_y$ because $B_y$ is a subspace. The first assertion of Lemma~\ref{lem:table} gives
\[
 (B_y+B_y)\cap(C+C)=\{0\}.
\]
Both sets have eight elements, so Lemma~\ref{lem:tiling} gives $B_y\oplus C=H$. For $(x,y)\in G$, this determines a unique $b\in B_y$ and $c\in C$ with $x=b+c$, and hence a unique representation $(x,y)=(b,y)+(c,0)$.
\end{proof}

\begin{proposition}[All differences of the domain]\label{prop:diff}
We have
\begin{equation}\label{eq:difference}
 E+E=((Z\cup\{0\})\times\{0\})
       \ \cup\ (H\times(Y\setminus\{0\})).
\end{equation}
Equivalently,
\begin{equation}\label{eq:missing}
 G\setminus(E+E)=(S\setminus\{0\})\times\{0\}.
\end{equation}
In particular, $|E+E|=36+64\cdot7=484$.
\end{proposition}
\begin{proof}
A sum of two points of $E$ has vertical coordinate zero exactly when the two points have the same label $y$. Its horizontal part then runs through $B_y+B_y=B_y$. The union identity in Lemma~\ref{lem:table} gives the first part of \eqref{eq:difference}.

Fix a nonzero $z\in Y$. The labels of the $A$ layers form
\[
 P=\{0,1,2,4,6\};
\]
the other three labels have $D$ layers. Translation by $z$ groups the eight labels into pairs $\{y,y+z\}$. If it preserved $P$, then $P$ would be a union of two-element pairs and would have even size. Since $|P|=5$, there is a $y\in P$ with $y+z\notin P$. For these labels, $B_y=A_i$ and $B_{y+z}=D_k$ for some $i,k$.

By Lemma~\ref{lem:table}, these two spaces intersect only at zero. All 64 sums $a+d$ with $a\in A_i$ and $d\in D_k$ are distinct: an equality would give $a+a'=d+d'$ in the intersection. Therefore $A_i+D_k=H$. Their two layers already contribute all of $H\times\{z\}$ to $E+E$. This works for every nonzero $z$ and proves the formula.
\end{proof}

\section{Every tiling complement}\label{sec:complements}
We first solve an eight-point problem in $H$. Its solution will also be used later for spectra.

\subsection{Pair-sums as edges}
An \emph{edge} on a set of vertices means an unordered pair of distinct vertices. A complete graph consists of all such pairs. For example, three vertices $a,b,c$ have the three edges $\{a,b\},\{a,c\},\{b,c\}$, a triangle. We use the eight points of $C$ as vertices and label an edge $\{c,c'\}$ by the vector $c+c'$.

\begin{lemma}\label{lem:relation}
The only nonempty even-sized subset of $C$ whose sum is zero is $C$ itself. Consequently each nonzero point of $S$ has a unique expression as the sum of an unordered pair of distinct points of $C$.
\end{lemma}
\begin{proof}
In a zero sum chosen from $e_1,\ldots,e_6,j$, compare coordinate $i$. The presence of $e_i$ must equal the presence of $j$. Thus either none of these seven vectors is chosen or all seven are chosen. Adding the zero vector does not change the sum. In the first case the only nonempty choice is $\{0\}$, of odd size; in the second case an even choice must include zero and is exactly $C$.

If two different unordered pairs had the same sum, cancel their common vertices. The remaining vertices would form a nonempty even zero-sum subset of size at most four. The first assertion forbids this. Existence of a pair representation follows from the definition of $S$.
\end{proof}

\begin{lemma}\label{lem:edges}
The labels of two distinct edges add to an element of $S$ if and only if the two edges share a vertex.
\end{lemma}
\begin{proof}
If the edges are $\{a,b\}$ and $\{a,c\}$, their labels add to $b+c\in S$. If they are disjoint, their labels add to the sum of four distinct vertices. The labels are distinct by Lemma~\ref{lem:relation}, so their sum is nonzero. Were that vector in $S$, it would therefore equal the sum of two distinct vertices. Canceling any common vertices between the four-element set and this pair gives a nonempty even zero-sum subset of $C$ of size at most six. It is nonempty because a four-element set is not a two-element set. Lemma~\ref{lem:relation} gives a contradiction.
\end{proof}

\begin{lemma}[The eight-point classification]\label{lem:Q}
If $Q\subseteq H$, $|Q|=8$, and $Q+Q\subseteq S$, then $Q$ is a translate of $C$.
\end{lemma}
\begin{proof}
Translate $Q$ by one of its points so that $0\in Q$; this does not change $Q+Q$, since the translating vector occurs twice in every sum. The seven nonzero points of $Q$ label seven distinct edges. The sum of any two belongs to $S$, so Lemma~\ref{lem:edges} says that these edges pairwise intersect.

Any family of at least four distinct pairwise intersecting edges has a common vertex. To prove this, take two edges $\{a,b\},\{a,c\}$. If some edge avoids $a$, it must be $\{b,c\}$ to meet both. An edge meeting all three either contains $a$, in which case it must be one of the first two, or avoids $a$, in which case it must be $\{b,c\}$. There can be no fourth distinct edge. Thus in our seven-edge family every edge contains the same vertex $c$.

There are exactly seven edges at $c$, one to each other point of $C$. Such a collection is called a \emph{star}. Its seven labels, together with zero, are $c+C$. This proves the assertion for the normalized $Q$, and undoing its initial translation proves the general statement.
\end{proof}

\subsection{Applying the classification}
\begin{proposition}\label{prop:allT}
Every tiling complement of $E$ is a translate of $T_0=C\times\{0\}$. There are exactly 512 of them, precisely eight contain zero, and none is an affine subspace.
\end{proposition}
\begin{proof}
Let $T$ be any complement. It has $512/64=8$ points. Translating it by a point of itself makes $0\in T$ and preserves the tiling property: translating all the sums translates the whole group. Every nonzero point of this normalized $T$ lies in $T+T$ and must avoid $E+E$. Formula~\eqref{eq:missing} forces it into $(S\setminus\{0\})\times\{0\}$. All differences of points of $T$ have the same restriction. Hence $T=Q\times\{0\}$ where $|Q|=8$ and $Q+Q\subseteq S$. Lemma~\ref{lem:Q} gives the asserted form after undoing the translation. Conversely, Proposition~\ref{prop:tile} shows that all translates of $T_0$ are complements.

We check that these translates are distinct. If $C+a=C$ with $a\ne0$, translation by $a$ would partition $C$ into four disjoint pairs, each having sum $a$. This contradicts uniqueness of pair-sums in Lemma~\ref{lem:relation}. Thus $C$ has no nonzero translation preserving it. The same holds for $T_0$: a translation preserving it must have zero vertical coordinate and then must preserve $C$. Therefore its 512 translates are distinct. A translate $T_0+g$ contains zero exactly when $g\in T_0$, giving eight normalized complements.

Finally, were $C$ an affine subspace, the fact that $0\in C$ would make it linear. Since it contains all six coordinate vectors, it would then equal $H$, contradicting $|C|=8$. If any translate of $T_0$ were affine, then so would be $T_0$, and its horizontal projection would make $C$ affine. This rules out every affine, hence every linear, complement.
\end{proof}
At this point the failure of the subspace-complement assertion in dimension nine has been proved without Fourier analysis. The remaining work identifies the spectra, rather than merely supplying one of them.

\section{Finite Fourier analysis from linear algebra}\label{sec:fourier}
We now develop the sign functions used to measure spectrality. All sums remain finite.

\subsection{Dot products, characters, and inversion}
For $x,\xi\in\F^n$, set $\xi\cdot x=\sum_i\xi_ix_i\in\F$. This is a bilinear form, but it is not a positive inner product: $(1,1)\cdot(1,1)=0$. It is \emph{nondegenerate}: if $x\ne0$, a coordinate vector $e_i$ with $x_i=1$ satisfies $e_i\cdot x=1$.

A linear functional is a linear map $\F^n\to\F$. It is determined by its values on the coordinate basis, so every such functional has the form $x\mapsto\xi\cdot x$ for a unique $\xi$. The associated \emph{character} is the real-valued function
\[
 \chi_\xi(x)=(-1)^{\xi\cdot x}.
\]
Its exponent is calculated modulo two and its value is the ordinary real number $1$ or $-1$. Bilinearity gives $\chi_\xi(x+x')=\chi_\xi(x)\chi_\xi(x')$ and $\chi_\xi\chi_\eta=\chi_{\xi+\eta}$. On $\F$, the two characters are the constant function $(1,1)$ and the function $(1,-1)$.

For a real- or complex-valued function $f$ on $\F^n$, define
\[
 W_f(\xi)=\sum_{x\in\F^n}f(x)(-1)^{\xi\cdot x},
 \qquad \widehat f(\xi)=2^{-n}W_f(\xi).
\]
These are the unnormalized and normalized \emph{Fourier transforms}, also called Walsh transforms in this setting. For a set $A$, its indicator $\one_A$ equals one on $A$ and zero elsewhere; we abbreviate $W_{\one_A}$ to $W_A$. The \emph{support} of a function is the set where it is nonzero. For example, on $\F$, $W_{\{0\}}=(1,1)$ whereas $W_{\F}=(2,0)$.

\begin{lemma}[Character orthogonality and inversion]\label{lem:inversion}
For $u\in\F^n$,
\[
 \sum_{x\in\F^n}(-1)^{u\cdot x}
 =\begin{cases}2^n,&u=0,\\0,&u\ne0.\end{cases}
\]
Consequently
\begin{equation}\label{eq:inversion}
 f(x)=2^{-n}\sum_{\xi\in\F^n}W_f(\xi)(-1)^{\xi\cdot x}.
\end{equation}
In particular, the transform determines the function uniquely.
\end{lemma}
\begin{proof}
If $u\ne0$, choose $i$ with $u_i=1$. Pair $x$ with $x+e_i$; their signs are opposite, so the sum is zero. For $u=0$ every term is one. To prove inversion, substitute the definition of $W_f$ into its proposed right side and interchange the finite sums. The coefficient of $f(t)$ becomes
\[
 2^{-n}\sum_\xi(-1)^{\xi\cdot(t+x)},
\]
which is one if $t=x$ and zero otherwise by the first identity.
\end{proof}

For a subspace $B\le\F^n$, its \emph{orthogonal complement}, also called its annihilator, is
\[
 B^\perp=\{\xi:\xi\cdot b=0\text{ for every }b\in B\}.
\]
For example, in $\F^2$ the line spanned by $(1,1)$ is its own orthogonal complement. In general $\dim B^\perp=n-\dim B$. Indeed, extend a basis of $B$ to a basis of the whole space. Every linear functional on $B$ extends by assigning arbitrary values on the added basis vectors. The restriction map from all linear functionals to those on $B$ is therefore onto, and its kernel is $B^\perp$. The dimension formula follows from rank-nullity. Nondegeneracy also gives $(B^\perp)^\perp=B$, by inclusion and equal dimensions.

\begin{lemma}\label{lem:subtransform}
For a subspace $B$ and a vector $a$,
\[
 W_B(\xi)=\begin{cases}|B|,&\xi\in B^\perp,\\0,&\xi\notin B^\perp,\end{cases}
 \qquad W_{a+B}(\xi)=(-1)^{\xi\cdot a}W_B(\xi).
\]
More generally, translation of a function multiplies its transform by the same character factor.
\end{lemma}
\begin{proof}
On $B$, a character annihilating $B$ is constantly one. Otherwise choose $b_0\in B$ with $\xi\cdot b_0=1$ and pair $b$ with $b+b_0$ to get zero. For a translated set, substitute $x=a+b$ in the defining sum. The same substitution proves the function statement.
\end{proof}

\begin{lemma}[Support and translation invariance]\label{lem:period}
For a subspace $U\le\F^n$, a function $f$ has Fourier support contained in $U$ if and only if
\[
 f(x+k)=f(x)\quad\text{for all }x\in\F^n,\ k\in U^\perp.
\]
In particular, a function on $H\times Y$ with Fourier support in $H\times\{0\}$ is independent of its $Y$ coordinate.
\end{lemma}
\begin{proof}
If the support is contained in $U$, every nonzero term in the inversion formula is unchanged by adding $k\in U^\perp$, proving invariance. Conversely, invariance and the translation formula give $W_f(\xi)=(-1)^{\xi\cdot k}W_f(\xi)$ for every such $k$. If $\xi\notin U=(U^\perp)^\perp$, some $k$ has $\xi\cdot k=1$, forcing $W_f(\xi)=0$. In the product case, $(H\times\{0\})^\perp=\{0\}\times Y$.
\end{proof}

\subsection{Spectra and their symmetric matrix criterion}
The complex-valued functions on a finite nonempty set $A$ form a vector space of dimension $|A|$. Use the inner product $\langle f,g\rangle=\sum_{a\in A}f(a)\overline{g(a)}$.
\begin{definition}
A set $\Lambda\subseteq\F^n$ is a \emph{spectrum} of $A$ if the restrictions of $\chi_\lambda$, $\lambda\in\Lambda$, form an orthogonal basis of this function space. A set admitting a spectrum is \emph{spectral}.
\end{definition}
For example, $\{00,10\}\subseteq\F^2$ has spectrum $\{00,10\}$: its two character columns are $(1,1)$ and $(1,-1)$. These spectra need not be unique; $\{00,11\}$ gives the same two columns on that domain.

\begin{lemma}[The character matrix]\label{lem:spectral}
Let $|A|=|\Lambda|=m$ and form the real matrix $H_{a,\lambda}=(-1)^{a\cdot\lambda}$. The following are equivalent:
\begin{enumerate}
\item $\Lambda$ is a spectrum of $A$;
\item $H^{\mathsf T}H=mI$;
\item $W_A(\lambda+\lambda')=0$ for distinct $\lambda,\lambda'\in\Lambda$.
\end{enumerate}
They also imply $HH^{\mathsf T}=mI$, so $A$ is a spectrum of $\Lambda$ and $W_\Lambda(a+a')=0$ for distinct $a,a'\in A$.
\end{lemma}
\begin{proof}
The $(\lambda,\lambda')$ entry of $H^{\mathsf T}H$ is $W_A(\lambda+\lambda')$. Each diagonal entry is $m$. Nonzero orthogonal columns are linearly independent, and $m$ such columns form a basis in dimension $m$, proving the equivalences. The equation gives the square real matrix $H$ a left inverse $m^{-1}H^{\mathsf T}$. A square matrix with a left inverse is invertible, so this is also its right inverse and $HH^{\mathsf T}=mI$. Its entries give the last assertions. This matrix multiplication is over the real numbers, not over $\F$.
\end{proof}

\begin{lemma}[Coordinate changes]\label{lem:affine}
If $M$ is an invertible linear map on $\F^n$, then $A\oplus T=\F^n$ implies $(MA+a)\oplus(MT+t)=\F^n$ for any $a,t$. If $\Lambda$ is a spectrum of $A$, then $M^{-\mathsf T}\Lambda+\eta$ is a spectrum of $MA+a$ for every $a,\eta$.
\end{lemma}
\begin{proof}
The tiling sums are the image under the bijection $x\mapsto Mx+a+t$ of the original sums. For spectrality,
\[
 (Mx+a)\cdot(M^{-\mathsf T}\lambda+\eta)
 =x\cdot\lambda+x\cdot M^{\mathsf T}\eta+a\cdot M^{-\mathsf T}\lambda+a\cdot\eta.
\]
Besides $x\cdot\lambda$, the terms depend only on the row, only on the column, or are constant. Thus the new character matrix is obtained by multiplying rows and columns by signs, which preserves their orthogonality. In particular, translating a domain or a spectrum is harmless. The inverse transpose, rather than $M$ itself, is essential for frequencies.
\end{proof}

\section{Every spectrum of the example}\label{sec:spectra}
We now prove the second half of Theorem~\ref{thm:example}. The argument has two steps: the difference set forces a product shape, and the horizontal layers restrict the remaining eight-point factor.

\subsection{An explicit spectrum}
The transform of the eight-point set $C$ is especially simple. For $u\in H$ of weight $w$,
\begin{equation}\label{eq:WC}
 W_C(u)=1+\sum_{i=1}^6(-1)^{u_i}+(-1)^{\sum_i u_i}
       =7-2w+(-1)^w.
\end{equation}
Its values for $w=0,1,\ldots,6$ are respectively
\[
 8,\ 4,\ 4,\ 0,\ 0,\ -4,\ -4.
\]
Consequently $\supp W_C=S$. For $\Lambda_0=C\times Y$, character orthogonality in $Y$ gives
\begin{equation}\label{eq:productW}
 W_{\Lambda_0}(u,v)=W_C(u)\sum_{y\in Y}(-1)^{v\cdot y}
 =\begin{cases}8W_C(u),&v=0,\\0,&v\ne0.\end{cases}
\end{equation}
By \eqref{eq:difference}, this vanishes at every nonzero difference of points of $E$. The matrix criterion, using row orthogonality, proves that $\Lambda_0$ is a spectrum of $E$. Both sets have 64 points, so the square-matrix hypothesis is satisfied.

\subsection{The annihilators of the six spaces}
To restrict an arbitrary spectrum, we need the union of $B_y^\perp$. A short linear map supplies it without a second table. Set $p(x)=\sum_i x_i\in\F$ and define
\[
 Jx=x+p(x)j\qquad(x\in H).
\]
Since $p(j)=0$ in dimension six, $p(Jx)=p(x)$ and $J^2x=x$. Thus $J$ is an invertible linear map. It fixes weight-four vectors and complements the six bits of a weight-three vector, again giving weight three. Therefore $J(Z)=Z$.

\begin{lemma}\label{lem:J}
If $B\le H$ has dimension three and every nonzero element of $B$ has weight three or four, then $B^\perp=J(B)$.
\end{lemma}
\begin{proof}
We claim that $b\cdot b'=p(b)p(b')$ for $b,b'\in B$. If one is zero this is immediate. If $b=b'$, both sides equal $p(b)$, since in $\F$ each coordinate and each parity bit equals its own square.

For distinct nonzero $b,b'$, their sum is again nonzero in $B$. The integer identity
\[
 \wt(b+b')=\wt(b)+\wt(b')-2r
\]
holds, where $r$ is the number of coordinates in which both are one. The dot product is $r$ modulo two. If both weights are three, their sum has even weight and hence weight four; thus $r=1$. If the weights are three and four, the sum has odd weight and hence weight three; thus $r=2$. If both are four, the sum has weight four and again $r=2$. These are exactly the parities $p(b)p(b')$.

It follows that $(Jb)\cdot b'=b\cdot b'+p(b)(j\cdot b')=0$. Hence $J(B)\subseteq B^\perp$. Both spaces have dimension three, giving equality.
\end{proof}
Apply $J$ to the union in Lemma~\ref{lem:table}. Because the repeated layers do not change that union, we obtain
\begin{equation}\label{eq:dualunion}
 \bigcup_{y\in Y}B_y^\perp=J\left(\bigcup_{y\in Y}B_y\right)=Z\cup\{0\}.
\end{equation}

\subsection{Completeness of the spectrum list}
\begin{proposition}\label{prop:allL}
Every spectrum of $E$ is $(C+a)\times Y$ for some $a\in H$. There are 64 distinct such sets and precisely eight contain zero.
\end{proposition}
\begin{proof}
Let $\Gamma$ be any spectrum of $E$. It has 64 points. The row-orthogonality conclusion of Lemma~\ref{lem:spectral} gives
\[
 W_\Gamma(d)=0\qquad(d\in(E+E)\setminus\{0\}).
\]
Using \eqref{eq:missing}, we get $\supp W_\Gamma\subseteq S\times\{0\}\subseteq H\times\{0\}$. By Lemma~\ref{lem:period}, $\one_\Gamma$ is independent of its vertical coordinate. Thus
\[
 \Gamma=Q\times Y\quad\text{with }|Q|=64/8=8.
\]
This product conclusion applies to every spectrum, irrespective of how it was obtained.

Take distinct $q,q'\in Q$ and put $u=q+q'\ne0$. For every $v\in Y$, $(u,v)$ is the difference of the two points $(q,0),(q',v)$ in $\Gamma$. Column orthogonality gives
\[
 0=W_E(u,v)=\sum_{y\in Y}(-1)^{v\cdot y}W_{B_y}(u)
 \qquad(v\in Y).
\]
These are all eight Fourier coefficients of the function $y\mapsto W_{B_y}(u)$ on $Y$. Inversion makes that function zero at every $y$. By Lemma~\ref{lem:subtransform}, $u\notin B_y^\perp$ for every $y$. Formula~\eqref{eq:dualunion} now forces $u\in S\setminus\{0\}$. Including the zero differences, $Q+Q\subseteq S$.

Lemma~\ref{lem:Q} gives $Q=C+a$. Conversely, every such product is a translate of the already verified spectrum $\Lambda_0$, so Lemma~\ref{lem:affine} proves that it is a spectrum. The 64 translates $C+a$ are distinct by the no-period argument in Proposition~\ref{prop:allT}. The product contains zero precisely when $a\in C$, yielding eight normalized spectra.
\end{proof}
Together with Proposition~\ref{prop:allT}, this proves Theorem~\ref{thm:example}. In particular, the spectrum classification is not a search through a sample of spectra: it excludes every other possibility analytically.

\subsection{An exact norm corollary}
For a subset $A\subseteq\F^n$, define its normalized \emph{Wiener norm} by
\[
 \W(A)=\sum_{\xi\in\F^n}|\widehat{\one_A}(\xi)|
       =2^{-n}\sum_\xi|W_A(\xi)|.
\]
This definition uses only a finite sum of absolute values. For a spectral set $A$, let $\wsp(A)$ be the minimum of $\W(\Lambda)$ over all its spectra. There are finitely many subsets of the ambient group, so this minimum exists.

\begin{corollary}\label{cor:norm}
Every spectrum of the explicit set $E$ has Wiener norm $15/8$. Consequently $\wsp(E)=15/8$.
\end{corollary}
\begin{proof}
Translation multiplies transform values by signs and hence preserves their absolute values. It suffices to use $\Lambda_0$. Equations~\eqref{eq:WC}--\eqref{eq:productW} give one transform value of magnitude 64, at zero, and 28 values of magnitude 32; all remaining values vanish. Therefore
\[
 \W(\Lambda_0)=\frac{64+28\cdot32}{512}=\frac{15}{8}.
\]
Proposition~\ref{prop:allL} classifies all spectra, so this calculation is the exact global minimum, not merely an upper bound from a selected spectrum. It is the norm of spectra, not the norm of the domain, that is being minimized.
\end{proof}

\subsection{Aperiodicity and the periodic tiling property}\label{sec:periodic}
The complement classification also excludes a weaker kind of algebraic
structure. A \emph{period} of a set $A$ is a vector $p$ with $A+p=A$.
The set is \emph{periodic} if it has a nonzero period, and
\emph{aperiodic} otherwise. For example, $\{00,10\}\subset\F^2$
has period $10$, whereas a singleton is aperiodic. Every nontrivial
linear subspace is periodic, since each of its vectors preserves it.

Following Fan and Zhang~\cite{FZ}, a finite abelian group has the
\emph{periodic tiling property} if every aperiodic tile has a periodic
tiling complement. This concerns all possible complements of a tile;
an aperiodic complement in one particular tiling does not disprove it.
Goyal~\cite{Goyal} emphasizes this property for a related example,
compared with ours in Appendix~\ref{app:comparison}.

\begin{corollary}\label{cor:PT}
The set $E$ and every tiling complement of $E$ are aperiodic.
Consequently $\F^9$ does not have the periodic tiling property.
\end{corollary}
\begin{proof}
Suppose $(u,v)\in H\times Y$ is a period of $E$. Comparing its layers gives
$B_y+u=B_{y+v}$ for every $y$. If a translate $B+u$ of a linear subspace
is another linear subspace $B'$, then $0\in B'$ implies $u\in B$,
so $B'=B$. Hence $B_y=B_{y+v}$ for every $y$.
The subspace $B_1=A_1$ occurs only once in \eqref{eq:layers}, so $v=0$.
Now $u\in B_y$ for every $y$. In particular,
$u\in A_0\cap D_0=\{0\}$ by Lemma~\ref{lem:table}. Thus $E$ is aperiodic.

Proposition~\ref{prop:allT} proves that every complement is a translate
of $C\times\{0\}$ and that this set has no nonzero period.
A translation preserves periods, since $(A+g)+p=A+g$ is equivalent
to $A+p=A$. Therefore every complement is aperiodic as well.
\end{proof}

\section{The positive results in small dimensions}
\label{sec:positive-elementary}

We now turn from the example to the positive direction.  The main point of
this section is that almost all of that direction can be proved with very
small configurations.  We give elementary proofs for every tile in
at most seven dimensions, and for every cardinality except sixteen in
eight dimensions.  The remaining sixteen-point assertion is identified
separately below; its proof in \cite{FI} uses the published classification
of perfect binary codes of length fifteen.

We recall two linear-algebra terms. The \emph{span} of a set is the set
of its linear combinations. The \emph{affine span} of a nonempty set $A$
is $a+\spanof(A+a)$ for any $a\in A$; it is the smallest affine subspace
containing $A$. For example, $\{00,10,01\}$ has affine span $\F^2$.
A \emph{linear complement} of $V\le G$ is a subspace $U$ such that
each vector of $G$ is uniquely $v+u$. One obtains it by extending a basis
of $V$ to a basis of $G$ and spanning the added vectors.

The following two Fourier consequences will be used below. Their proofs
are included to keep the small-complement argument independent of any
general theorem about spectral sets.
\begin{lemma}\label{lem:tilingfourier}
If $A\oplus T=\F^n$, then $W_A(u)W_T(u)=0$ for $u\ne0$.
If $K\le\F^n$, then $A\oplus K=\F^n$ if and only if $K^\perp$
is a spectrum of $A$.
\end{lemma}
\begin{proof}
Unique representation and the character product identity give
\[
 W_A(u)W_T(u)=\sum_{a\in A,t\in T}(-1)^{u\cdot(a+t)}
            =\sum_{x\in\F^n}(-1)^{u\cdot x}=0\quad(u\ne0).
\]
If $A\oplus K=\F^n$, then $|K^\perp|=2^n/|K|=|A|$.
For nonzero $u\in K^\perp$, the identity above and $W_K(u)=|K|$
give $W_A(u)=0$. The difference of two distinct points of $K^\perp$
is such a $u$, so Lemma~\ref{lem:spectral} proves spectrality.
Conversely, if $K^\perp$ is a spectrum of $A$, row orthogonality gives
$W_{K^\perp}(a+a')=0$ for distinct $a,a'\in A$. Its transform is
nonzero exactly on $(K^\perp)^\perp=K$. Thus $(A+A)\cap K=\{0\}$,
and the spectral cardinality gives $|A||K|=2^n$. Apply
Lemma~\ref{lem:tiling}.
\end{proof}

\subsection{A seven-coordinate auxiliary problem}

We first prove a small fact about subsets of $\mathbb F_2^7$.  It is often
phrased in the language of error-correcting codes, but all we need is a
partition by translates of the following eight-point set:
\[
 R=\{0,e_1,\ldots,e_7\}.
\]
The Hamming distance between two vectors is the number of coordinates
where they differ, namely the Hamming weight of their sum.  Thus $R$ is
exactly the set of vectors at distance at most one from zero.  For example,
$e_1+e_2$ has distance two from zero and distance one from $e_1$.
A set $D$ with $R\oplus D=\mathbb F_2^7$ is called a
\emph{perfect one-error-correcting code}: every vector lies at distance
at most one from exactly one point of $D$.  The terminology will not be
needed in the argument.

\begin{lemma}\label{lem:perfect-seven-elementary}
If $R\oplus D=\mathbb F_2^7$, then $D$ is an affine subspace of dimension
four.  More precisely, after translating $D$ to contain zero, there is a
labeling $p_1,\ldots,p_7$ of the seven nonzero vectors of $\mathbb F_2^3$
such that
\[
 D=\left\{x\in\mathbb F_2^7:\sum_{i=1}^7 x_i p_i=0\right\}.
\]
\end{lemma}

The idea is that the weight-three members of $D$ determine all of $D$.
They form a small arrangement of seven triples, and the partition property
forces this arrangement uniquely up to relabeling.

\begin{proof}
Translate $D$ so that $0\in D$.  Counting the partition gives $|D|=16$.
Distinct points of $D$ have distance at least three: if two centers differ
in one coordinate, one center lies in both of their radius-one balls; if
they differ in two coordinates, changing just one of those coordinates
gives a vector at distance one from both.  Either possibility contradicts
the partition by the translates $R+d$.  In particular, $D$ contains no vectors of weight one or two.

Consider a vector of weight two, with support $\{i,j\}$.  It must be at
distance one from a member of $D$ of weight three: the only other possible
weights are one and two, already excluded.  Consequently there is exactly
one triple $L\subseteq\{1,\ldots,7\}$ containing the pair $\{i,j\}$
whose indicator vector belongs to $D$.  Call these triples \emph{lines}.
There are $\binom72/\binom32=7$ lines.  Two different lines share at
most one point, since a pair cannot occur twice.

Here is an explicit determination of all seven lines.  Label one line
$123$, choose a fourth point $4$, and call the third points on the lines
through $14$ and $24$ respectively $5$ and $6$.  These are distinct new
points: every other choice repeats a pair already used.  The line through
$34$ must then have the remaining point $7$ as its third point.  The same
no-repeated-pair rule forces the remaining lines.  The resulting list is
\[
 123,\quad145,\quad246,\quad347,\quad167,\quad257,\quad356.       \tag{*}
\]
For example, the third point on the line through $16$ cannot be $2,3,4,$
or $5$, since each would reuse a pair containing $1$; it must be $7$.
The lines through $25$ and $35$ are then forced in the same way.
The displayed seven triples use every pair exactly once, so the list is
complete.

Every point belongs to three lines: its six pairs with the other points
are grouped two at a time.  A fixed line therefore meets the six other
lines, two at each of its three points.  In particular, any two lines
intersect.

A set of four coordinates that contains a line cannot be the support of
a member of $D$, since its distance from that line's indicator is one.
There are $7\cdot4=28$ such four-sets: a four-set cannot contain two lines,
as two lines have a five-point union.  The other seven four-sets are
exactly the complements of the seven lines.  Indeed, a line complement
cannot contain a line, because any two lines meet.

There are $\binom73-7=28$ triples that are not lines.  The indicator of
any such triple must be at distance one from a weight-four member of $D$;
a weight-two member is impossible and the triple itself is not in $D$.
Each allowable four-set contains four such triples.  Two different
allowable four-sets intersect in two points, because their complementary
lines have a five-point union.  Thus each nonline triple belongs to
exactly one allowable four-set, and all seven allowable four-sets are
forced to occur in $D$.

We have found fifteen points of $D$: zero, the seven line indicators,
and their seven complements.  No weight-five point can be added: every
five-set contains a line (otherwise each of its five four-subsets would
be a line complement, whereas two line complements meet in only two
points), and its distance from that line indicator would be two.
No weight-six point can be added either: if its missing coordinate is
$i$, choose a line through $i$; that line's complement lies inside the
six-set, at distance two.  The last point must therefore be the all-ones
vector.

Finally, give the points the labels
\[
 p_1=a,\quad p_2=b,\quad p_3=a+b,\quad p_4=c,\quad
 p_5=a+c,\quad p_6=b+c,\quad p_7=a+b+c,
\]
where $a,b,c$ are the coordinate vectors of $\mathbb F_2^3$.
Each line in (*) has label sum zero.  The sum of all seven labels is
also zero, because each coordinate occurs four times.  Hence all sixteen
points found above lie in the kernel of the displayed linear map.
The map is onto, so its kernel has dimension $7-3=4$ and contains exactly
sixteen points.  This proves the claimed equality.
\end{proof}

\begin{lemma}\label{lem:small-tile-elementary}
Every tile with at most eight points in a vector space over $\F$ has a
linear-subspace tiling complement.
\end{lemma}

This is the small-tile result of Cohen, Litsyn, Vardy, and Z\'emor
\cite{CLVZ}.  The following proof supplies the details needed here,
using the preceding seven-coordinate calculation rather than an external
classification.

\begin{proof}
The size of a tile divides the size of its ambient space, so the possible
sizes are $1,2,4,8$.  Translate the tile $A$ to contain zero and work first
in its linear span $V$.  If $A\oplus T=G$, then
$A\oplus(T\cap V)=V$, since both $x$ and its summand from $A$ lie in $V$
whenever $x\in V$.  A subspace complement in $V$ can be enlarged by any
linear complement of $V$ in $G$.

For a one-point set in its span, the span is $\{0\}$ and its
complement there is $\{0\}$. For a two-point set containing zero, its
span is the set itself, again with complement $\{0\}$.  A four-point set is
either a two-dimensional subspace or has the form $\{0,a,b,c\}$ with
$a,b,c$ independent.  In the latter case, map $a,b,c$ to the three nonzero
vectors of $\mathbb F_2^2$ and extend linearly.  This map is a bijection
on $A$, so its kernel is a subspace complement.

Now suppose $A=\{0,a_1,\ldots,a_7\}$ has eight points.  Define an onto
linear map $F:\mathbb F_2^7\to V$ by $F(e_i)=a_i$, and put
$D=F^{-1}(T\cap V)$.  The unique decomposition of $F(x)$ in
$A\oplus(T\cap V)$ selects a unique $r\in R$ for which $x+r\in D$.
Thus $R\oplus D=\mathbb F_2^7$.  By
Lemma~\ref{lem:perfect-seven-elementary}, $D$ is a coset of a
four-dimensional subspace $L$.  For $k\in\ker F$ we have $D+k=D$, since $F(x+k)=F(x)$.
A coset $d_0+L$ is preserved by translation by $k$ exactly when $k\in L$.
Thus $\ker F\subseteq L$.  Translating $D$ to $L$ preserves
the tiling with $R$.  It follows that $F(L)$ complements $A$ in $V$.
For completeness, if $F(r)=F(r')+F(\ell)$ with $r,r'\in R$ and
$\ell\in L$, then $r+r'+\ell\in\ker F\subseteq L$; uniqueness in
$R\oplus L$ gives $r=r'$.  Also
$|A||F(L)|=8\cdot2^{4-(7-\dim V)}=|V|$, proving the tiling.
\end{proof}

\subsection{Small tiling complements}

The roles of a tile and its complement are not interchangeable when the
conclusion asks that one specified factor be a subspace.  The next result
therefore requires a separate argument.

\begin{lemma}\label{lem:common-complement}
Two subspaces $U,V$ of the same dimension in a finite-dimensional vector
space have a common linear complement.
\end{lemma}
\begin{proof}
Let $D=U\cap V$, choose a basis of $D$, and extend it by
$u_1,\ldots,u_k$ to a basis of $U$ and by $v_1,\ldots,v_k$ to a basis
of $V$.  The union of these three lists is a basis of $U+V$.  The space
$L=\langle u_1+v_1,\ldots,u_k+v_k\rangle$ meets both $U$ and $V$
only in zero, as comparison of the coefficients of the $v_i$, or of the
$u_i$, shows.  If $R_0$ is any linear complement of $U+V$ in the ambient
space, then $L+R_0$ is disjoint from both $U$ and $V$ and has the required
complementary dimension.
\end{proof}

We also need the elementary Fourier-support observation of
Ferguson and Iosevich \cite[Lemma 2.3]{FI}.  Recall that the support
of a function is the set where it is nonzero.

\begin{lemma}\label{lem:support-plane-elementary}
If $\varnothing\ne A\ne\mathbb F_2^n$ and $|A|\ne2^{n-1}$, then
$\operatorname{supp}W_A$ contains a two-dimensional subspace.
\end{lemma}
\begin{proof}
Write $a(u)=2^{-n}W_A(u)$ and $\delta=a(0)=|A|/2^n$.  Fourier inversion
and the identity $\mathbf1_A^2=\mathbf1_A$ give
\[
 a(s)=\sum_t a(t)a(s+t).
\]
Indeed, square $\mathbf1_A(x)=\sum_u a(u)(-1)^{u\cdot x}$, group terms
with the same sum of indices, and compare Fourier coefficients.  Since
$\mathbf1_A$ is not constant, there is an $s\ne0$ with $a(s)\ne0$.
Separating the terms $t=0,s$ gives
\[
 (1-2\delta)a(s)=\sum_{t\notin\{0,s\}}a(t)a(s+t).
\]
The left side is nonzero.  Therefore one summand on the right is nonzero,
and its index $t$ gives the plane $\{0,s,t,s+t\}$ inside the support.
\end{proof}

\begin{proposition}\label{prop:small-partner-elementary}
Suppose $E\oplus T=\mathbb F_2^d$, where $d\le8$ and $|T|\le8$.
Then $E$ has a linear-subspace tiling complement.
\end{proposition}

The only substantial case is $|T|=8$.  We organize it by the dimension of
$\langle T\rangle$.  In dimensions six and seven the number of cosets in
the ambient space is small enough to choose a single complement that
works in all of them.  This is precisely where the bound $d\le8$ enters.

\begin{proof}
Translate $T$ to contain zero.  If $|T|=1$ or $2$, it is itself a
subspace.  If $|T|=4$, choose independent $a,b\in T$.  Then
$P=\{0,a,b,a+b\}\subseteq T+T$.  Hence $(E+E)\cap P=\{0\}$, and
$|E||P|=2^d$ proves $E\oplus P=\mathbb F_2^d$.

Suppose henceforth that $|T|=8$.  Put $V=\langle T\rangle$ and
$r=\dim V$, so $3\le r\le7$.  We use the character group $V^*$,
meaning the vector space of linear maps $\alpha:V\to\mathbb F_2$.
The restriction map $\rho:(\mathbb F_2^d)^*\to V^*$ is onto with
kernel of dimension $d-r$.  If a subspace $L\le V^*$ of dimension
$r-3$ lies in $\operatorname{supp}W_T$, then $\rho^{-1}(L)$ has
dimension $d-3$ and is a spectrum of $E$: the tiling Fourier identity
makes $W_E$ vanish at each of its nonzero points.  The annihilator of
this subspace spectrum is the required three-dimensional complement.

For $r=3$, take $L=\{0\}$.  For $r=4$, the nonconstant indicator of $T$
has a nonzero Fourier coefficient at some nonzero $u$; take
$L=\{0,u\}$.  For $r=5$, apply
Lemma~\ref{lem:support-plane-elementary}, since $|T|=8\ne16$.

An \emph{affine dependency} among points $t_i$ means coefficients
$\varepsilon_i\in\F$, not all zero, with
$\sum_i\varepsilon_i=0$ and $\sum_i\varepsilon_i t_i=0$.
For example, the four points $0,e_1,e_2,e_1+e_2$ have such a dependency
using all four coefficients equal to one. Points are \emph{affinely
independent} when no such dependency exists. Equivalently, their
augmented columns $(1,t_i)$ are linearly independent.

For $r=6$, the eight points of $T$ have exactly one nonzero affine
dependency.  To see this, form the seven-by-eight matrix with columns
$(1,t)$ for $t\in T$; its rank is $7$, so its nullspace has dimension
one.  Its nonzero vector selects an even number of points with zero sum.
Two selected points are impossible because the points are distinct.
Thus the dependency has length $4$, $6$, or $8$.

If the length is $4$ or $6$, choose the origin to be a point outside the
dependency and remove one point of the dependency.  The remaining seven
points are affinely independent; their six nonzero points can be taken
as a basis.  In suitable coordinates the set is therefore
\[
 \{0,e_1,\ldots,e_6,e_1+\cdots+e_k\},\qquad k=3\text{ or }5.
\]
Its transform at $u$ is
\[
 W_T(u)=7-2\operatorname{wt}(u)+(-1)^{u_1+\cdots+u_k}.
\]
For $k=3$, take $L=\langle e_1,e_2,e_4\rangle$ in the dual
coordinates; for $k=5$, take $L=\langle e_1,e_2,e_6\rangle$.
In the order $0,a,b,a+b,c,a+c,b+c,a+b+c$ of these three-generator
spaces, the transform values are respectively
\[
 8,\ 4,\ 4,\ 4,\ 6,\ 2,\ 2,\ 2.
\]
All are nonzero, so the preceding subspace-spectrum argument applies.

If the dependency has length eight, translate one point to zero and
remove one other point. The six remaining nonzero vectors form a basis:
a dependence among the seven retained augmented columns would omit the
removed point, whereas the unique nonzero dependency uses all eight.
The removed point is the sum of those six basis vectors, so
\[
 T=C=\{0,e_1,\ldots,e_6,e_1+\cdots+e_6\}\subset V.
\]
In each coset $v+V$ of $V$, translate the slice $E\cap(v+V)$ back into
$V$ and call it $A_v$.  Then $A_v\oplus C=V$ and $|A_v|=8$.
Let $F:\mathbb F_2^7\to V$ send the first six coordinate vectors to
$e_1,\ldots,e_6$ and the seventh to $e_1+\cdots+e_6$.
The same lifting argument as in
Lemma~\ref{lem:small-tile-elementary} shows that
$R\oplus F^{-1}(A_v)=\mathbb F_2^7$.
Lemma~\ref{lem:perfect-seven-elementary} makes each $A_v$ an affine
three-dimensional subspace, say $b_v+B_v$.

Let $Z=V\setminus(C+C)$, the 35 vectors of weights three or four.
Since $A_v$ tiles with $C$, its seven nonzero differences belong to $Z$;
that is, $B_v\setminus\{0\}\subseteq Z$.  There are
$2^{d-6}\le4$ slices, so
\[
 \left|\bigcup_v(B_v\setminus\{0\})\right|\le4\cdot7=28<35.
\]
Choose $z\in Z$ outside this union.  If $z$ has weight three, let $K$
be the coordinate space on its three nonzero positions.  If $z$ has
weight four, let $i,j$ be its two zero positions and put
$K=\langle z,e_i,e_j\rangle$.  In either case $\dim K=3$ and
$K\cap Z=\{z\}$.  In the second case, the other six nonzero vectors
have weights $1,1,2,5,5,6$; the first case is equally immediate from the
coordinate description.  Hence $K$ meets every $B_v$ only in zero.
A difference of points from different slices is not in $V$, so it is not
in $K$ either.  Thus $(E+E)\cap K=\{0\}$, proving that $K$ complements
$E$.

Finally suppose $r=7$.  Then $T$ consists of zero and a basis of $V$.
Each slice of $E$ in a coset of $V$ is a translate of a four-dimensional
subspace by Lemma~\ref{lem:perfect-seven-elementary}.  Since $d\le8$,
there are at most two slices.  Apply
Lemma~\ref{lem:common-complement} inside $V$ to their two direction
spaces (or choose any complement if there is only one).  The resulting
three-dimensional space $K\le V$ avoids all nonzero within-slice
differences; cross-slice differences lie outside $V$.  Therefore $K$
complements $E$, as required.
\end{proof}

\begin{corollary}\label{cor:elementary-positive-range}
Every tile in $\mathbb F_2^d$ with $d\le7$ has a linear-subspace
complement.  In $\mathbb F_2^8$, the same conclusion holds for every
tile whose cardinality is not sixteen.  It also holds for a sixteen-point
tile whose affine span has dimension at most seven.
\end{corollary}
\begin{proof}
A tile and its complement have power-of-two cardinalities with product
$2^d$.  For $d\le7$, at least one factor has at most eight points.
If the tile does, use Lemma~\ref{lem:small-tile-elementary}; if its
complement does, use Proposition~\ref{prop:small-partner-elementary}.
For $d=8$ the only case not covered is a factorization $16\cdot16$.
If a sixteen-point tile lies in an affine space of dimension at most
seven, translate and restrict to its span, use the first assertion, and
extend the resulting subspace complement to the whole space.
\end{proof}

\subsection{The remaining sixteen-point assertion}
The elementary arguments above leave a tile of size sixteen whose affine
span has dimension eight. We state the result needed for this case
and explain its dependence on the classification of perfect codes.

\begin{theorem}[External input: Ferguson--Iosevich]\label{thm:external16}
Every sixteen-point tile of $\F^8$ has a four-dimensional linear-subspace
tiling complement.
\end{theorem}
This is \cite[Lemma~6.5]{FI}, with its finite input in
\cite[Lemma~6.4]{FI}. It is a component of Theorem~1.1 of that preprint, which gives
the equivalence of tiling, spectrality, and a subspace complement in
dimensions seven and eight. The latter full equivalence is cited as
background; only Theorem~\ref{thm:external16} is needed as an external
positive result in the present proof.

Here is the exact nature of the external step. Replace the seven-coordinate
set $R$ used above by $R_{15}=\{0,e_1,\ldots,e_{15}\}\subseteq\F^{15}$.
For a sixteen-point tile $A=\{0,a_1,\ldots,a_{15}\}$ spanning $\F^8$,
the linear map $F(e_i)=a_i$ lifts a tiling complement $T$ to a partition
$R_{15}\oplus F^{-1}(T)=\F^{15}$, by precisely the unique-decomposition
argument already used for $R$. The lifted set is a perfect code of length
fifteen. Its \emph{translation kernel} is
\[
 M=\{v:F^{-1}(T)+v=F^{-1}(T)\}.
\]
This is a subspace: zero preserves the set, and the sum of two preserving
translations preserves it. It contains $\ker F$, of dimension $15-8=7$.
The proof in \cite{FI} uses the complete classification of perfect codes of
length fifteen by \"Osterg\aa rd and Pottonen~\cite{OP}. Among the 5983
classes, 69 have a translation kernel of dimension at least seven; quotient
certificates supply subspace complements for all of these possibilities.
The classification's completeness is an external input, not a consequence
of verifying the individual certificates.

The elementary seven-coordinate proof above does not extend automatically
to fifteen coordinates: its assertion that every perfect code is affine
is false there. We do not reproduce the length-fifteen classification or
claim a replacement proof. Theorem~\ref{thm:external16} is the only
positive result whose proof is not reproduced here.
Question~8.3 of \cite{FI} asks for a proof of this assertion without
the complete classification.

\begin{corollary}\label{cor:positive8}
Every nonempty tile in $\F^d$ with $d\le8$ has a linear-subspace complement.
\end{corollary}
\begin{proof}
For $d\le7$, use Corollary~\ref{cor:elementary-positive-range}.
In dimension eight, that corollary handles every size except sixteen,
and Theorem~\ref{thm:external16} handles the remaining size.
\end{proof}

Corollaries~\ref{cor:positive8} and~\ref{cor:PT} also show that nine
is the least positive dimension in which the periodic tiling property
fails over $\F$. Indeed, for $1\le d\le8$, an aperiodic tile $A$ is not
the whole group, since the whole group has nonzero periods. Its subspace
complement therefore has at least two elements and is periodic.
This minimality statement uses the same external sixteen-point input.

\subsection{Further elementary reductions of the remaining case}
The external assertion can be narrowed further. These reductions do not
replace it, but identify configurations a direct proof would still need
to handle. We use periods and aperiodicity as defined in
Section~\ref{sec:periodic}.

\begin{proposition}\label{prop:16reductions}
Suppose $A\oplus T=\F^8$ and $|A|=|T|=16$. Then $A$ has a subspace
complement if any of the following holds:
\begin{enumerate}
\item $A$ has affine span of dimension at most seven;
\item $T$ has affine span of dimension at most six;
\item either factor has a nonzero period;
\item the sum of all points of either factor is nonzero.
\end{enumerate}
\end{proposition}
\begin{proof}
The first assertion is Corollary~\ref{cor:elementary-positive-range}.
For the second, translate $T$ into its span $V$, of dimension $r$ with
$4\le r\le6$. We seek a subspace of dimension $r-4$ in its Fourier
support on $V$. For $r=4$, take zero. For $r=5$, a nonzero coefficient
away from zero exists because the indicator is not constant; its span
is a line. For $r=6$, Lemma~\ref{lem:support-plane-elementary} gives a
plane, since $|T|=16\ne32$. Lift this frequency subspace by restriction
from the ambient dual. The lift has dimension $(8-r)+(r-4)=4$, lies
in the support of $W_T$, and is a spectrum of $A$ by
Lemma~\ref{lem:tilingfourier}. Its annihilator is a subspace complement.

For the third assertion let $p\ne0$ be a period and let
$q:\F^8\to\F^8/\langle p\rangle$ be the quotient map. The quotient
is a seven-dimensional vector space over $\F$, identified with $\F^7$
by any choice of basis. If $p$ is a period of $T$, then $p\in T+T$
and $p\notin A+A$. Thus $q$ is injective on $A$ and
$q(A)\oplus q(T)=\F^7$. The sizes are sixteen and eight: the fibers
of $T$ have size two, while those of $A$ have size one. The quotient
sums cover $\F^7$, and there are exactly $16\cdot8=128$ pairs, so
every quotient representation is unique. The dimension-seven
result supplies a subspace complement $L$ of $q(A)$. The inverse image
$q^{-1}(L)$ is a subspace complement of $A$: the unique quotient
representation selects a unique point of $A$, and its remaining difference
belongs to the inverse image.

If $p$ is a period of $A$, interchange the roles to obtain
$q(A)\oplus q(T)=\F^7$ with sizes eight and sixteen. Take a
four-dimensional complement $L$ of $q(A)$ and lift a basis of $L$
to $\F^8$. Its span $K$ projects bijectively onto $L$, so
$K\cap\langle p\rangle=\{0\}$. A difference in $(A+A)\cap K$
projects into $(q(A)+q(A))\cap L=\{0\}$ and hence is zero.
The cardinalities now prove $A\oplus K=\F^8$.

For the fourth assertion put $\sigma_A=\sum_{a\in A}a$ in $\F^8$.
For a frequency $u$, if $m_u$ counts the points $a$ with $u\cdot a=1$,
then $W_A(u)=16-2m_u$ and $m_u\equiv u\cdot\sigma_A\pmod2$.
Consequently
\[
 W_A(u)\equiv-2(u\cdot\sigma_A)\pmod4.
\]
If $\sigma_A\ne0$, this transform is nonzero whenever
$u\cdot\sigma_A=1$. The tiling identity forces $W_T$ to vanish at
all those frequencies. Its support therefore lies in
$\langle\sigma_A\rangle^\perp$. Lemma~\ref{lem:period} makes
$\sigma_A$ a nonzero period of $T$, reducing to the third assertion.
The same argument with the factors interchanged applies if the sum of
the points of $T$ is nonzero.
\end{proof}
Thus a direct proof may restrict attention to aperiodic $16$-by-$16$
tilings with zero sums of the points of both factors, where the proposed
tile spans eight dimensions and its complement spans seven or eight.
We make no claim that these restrictions alone settle the case.

\section{The threshold in every dimension}\label{sec:lifting}
We finish the proof of Theorem~\ref{thm:threshold}. It is not enough to
have a counterexample in dimension nine: the universal assertion might,
without a further argument, conceivably hold again in a larger dimension.
The following product lemma prevents that.

\begin{lemma}[Products with a whole vector space]\label{lem:product}
Let $A\subseteq V=\F^n$ be nonempty, let $U=\F^k$, and put
$A^*=A\times U\subseteq V\times U$.
\begin{enumerate}
\item If $A\oplus T=V$, then $A^*\oplus(T\times\{0\})=V\times U$.
\item If $A^*$ has a linear-subspace tiling complement, then $A$ has one.
\item If $\Lambda$ is a spectrum of $A$, then $\Lambda\times U$ is a
spectrum of $A^*$, with $U$ identified with its frequency space by the
dot product.
\end{enumerate}
\end{lemma}
\begin{proof}
For the first assertion, uniquely write $x=a+t$ in $V$; then
$(x,u)=(a,u)+(t,0)$ is the unique product decomposition.

For the second, let $K\le V\times U$ complement $A^*$. Its difference
set is
\[
 A^*+A^*=(A+A)\times U.
\]
It contains $\{0\}\times U$, since $A$ is nonempty. The tiling
criterion therefore gives $K\cap(\{0\}\times U)=\{0\}$.
Horizontal projection $\pi:V\times U\to V$ is injective on $K$, so
$\pi(K)$ is a subspace of the same cardinality as $K$. If a nonzero
$x\in\pi(K)$ belonged to $A+A$, its unique lift $(x,u)\in K$ would
belong to $(A+A)\times U$, a contradiction. Thus
\[
 (A+A)\cap\pi(K)=\{0\},\qquad
 |A||\pi(K)|=|A||K|=|V|.
\]
The last equality follows from $|A||U||K|=|V||U|$.
Lemma~\ref{lem:tiling} now proves that $\pi(K)$ complements $A$.

For the third assertion, the inner product of the characters indexed by
$(\lambda,v)$ and $(\lambda',v')$ on $A\times U$ is
\[
 \left(\sum_{a\in A}(-1)^{(\lambda+\lambda')\cdot a}\right)
 \left(\sum_{u\in U}(-1)^{(v+v')\cdot u}\right).
\]
If the two frequencies differ, either the first factor vanishes by
spectrality or the second vanishes by character orthogonality.
The number of characters is $|\Lambda||U|=|A||U|$, as required.
\end{proof}

This lemma also explains directly how the dimension-eight positive
statement implies all lower-dimensional ones. If $A\subseteq\F^d$
tiles with $d<8$, then $A\times\F^{8-d}$ tiles $\F^8$ and has a
subspace complement by Corollary~\ref{cor:positive8}. Project that
complement using the second assertion to obtain a subspace complement
of $A$. This is an alternative to the elementary lower-dimensional proof
already given.

\begin{proof}[Proof of Theorem~\ref{thm:threshold}]
For $d\le8$, Corollary~\ref{cor:positive8} proves the assertion,
with precisely the external input stated in
Theorem~\ref{thm:external16}. For $d\ge9$, set
\[
 E_d=E\times\F^{d-9}.
\]
Theorem~\ref{thm:example} and Lemma~\ref{lem:product} show that $E_d$
is a tile and is spectral. If it had a linear-subspace complement, the
same lemma would give one for $E$, contradicting
Proposition~\ref{prop:allT}. Thus the universal subspace-complement
property fails in every dimension at least nine.
\end{proof}
The larger examples have $|E_d|=64\cdot2^{d-9}=2^{d-3}$, and every
tiling complement has eight points. In particular, any subspace
complement would have dimension three; the projection in the proof
preserves that dimension.

\section{Exact verification and the role of computation}\label{sec:verification}
The construction arose from a finite computational search. Its presentation
here was chosen so that Table~\ref{tab:spaces} and the odd-cardinality
argument give certificates that can be checked without software. The
proofs establish the complete classifications for this particular domain;
no classification of all spectral tiles is claimed.

The proof of Theorem~\ref{thm:example} does not depend on running a
program. Its finite data are the six rows of Table~\ref{tab:spaces};
the union count, cross intersections, and annihilator identity were all
justified in the text. A separate standard-library Python verifier is
included among the ancillary files to check the transcription and the
exact calculations.

Download \texttt{verify\_example.py} and \texttt{example.json} from the
ancillary files into the same directory, and run
\begin{center}\texttt{python3 verify\_example.py}\end{center}
there. Python 3.10 or later
suffices. The code uses integer arithmetic, with vector addition implemented
by XOR and $x\cdot y$ implemented by the parity of the number of one-bits
in the bitwise AND of $x$ and $y$. There are no random choices, external
solvers, search cutoffs, or network dependencies. The data include every
point of $E$, $T_0$, and $\Lambda_0$, as well as the bases and layer
assignment. The output includes every tiling complement and every spectrum.

The verifier checks:
\begin{enumerate}
\item the point counts, bases, layer assignment, both union identities,
and exact difference set;
\item every unique representation in the displayed tiling and every
entry of the integer Gram matrix for the displayed spectrum;
\item the eight normalized complements by an exhaustive increasing-order
search, and every one of their 512 distinct ambient translates;
\item the eight normalized horizontal spectrum factors, their 64 distinct
translates, and the integer Gram matrix and exact rational Wiener norm
for each resulting spectrum;
\item all three-dimensional subspaces of $\F^9$, none of which is a
complement.
\end{enumerate}
For the last check, the total number is
\[
 \frac{(512-1)(512-2)(512-4)}{(8-1)(8-2)(8-4)}=788{,}035.
\]
The numerator counts ordered independent triples in the ambient space:
a second vector avoids the two-point span of the first, and a third
avoids their four-point span. The denominator counts the ordered bases
inside one three-dimensional space. The program enumerates the unique
reduced-row-echelon basis of each subspace, allowing every assignment to
its free entries, and checks that the resulting total agrees with this
count. Row reduction over a field gives each subspace exactly one such
basis; thus the coverage has an independent linear-algebra justification.

The normalized eight-point searches each visit 125 recursion nodes and
find exactly eight sets. A set can be translated to contain zero without
changing its pairwise sums. The search then chooses its remaining points
in increasing order, retaining a new point only when all its differences
from previously chosen points are allowed. This includes every admissible
normalized set exactly once. The spectrum search is exhaustive only
\emph{after} the analytic reduction of Proposition~\ref{prop:allL}; it
is not represented as an enumeration of all 64-point subsets of $G$.

The verifier passes with both the ordinary interpreter and \texttt{python -O}. In total,
65 integer Gram matrices are checked (one initial check and one for each
of the 64 listed spectra), and all 788,035 subspaces are checked. The
program's explicit verification conditions remain active under
\texttt{python -O}. These computations independently check the individual
example and the stated finite enumerations. They neither establish a
classification of all nine-dimensional tiles nor replace the external
positive theorem in dimension eight.

The additional verifier \texttt{verify\_comparison.py}, together with
\texttt{comparison\_data.json} and the two files used above, checks the
explicit tables and structural hypotheses in Appendix~\ref{app:comparison}.
Run \texttt{python3 verify\_comparison.py} from their common directory.
It checks the common difference set, the annihilator union, aperiodicity,
all 512 complements and all 64 displayed spectra of Goyal's domain,
their exact norms, and the layer-direction multiplicities distinguishing
the domains. Exhaustiveness follows from the theorem in the appendix.
These are exact Python checks, not a replay of Goyal's Lean development.

\section{Further questions}
The construction leaves several focused problems.
\begin{question}
Can Theorem~\ref{thm:external16}, for sixteen-point tiles in $\F^8$,
be proved without the classification of perfect codes of length fifteen?
\end{question}
An affirmative answer would make the sharp-threshold proof given here
entirely elementary and self-contained. The elementary reductions in
Proposition~\ref{prop:16reductions} specify the remaining configurations;
the small-complement proof already avoids the order-32 sign-matrix
classification for the tiling direction. This question is the sixteen-point
part of Question~8.3 in our earlier preprint~\cite{FI}.

\begin{question}
Does every 64-point spectral subset of $\F^9$ tile?
\end{question}
Our example answers the stronger subspace-complement question negatively,
but tiles explicitly. It therefore supplies no negative answer to this
question. The distinction between a general complement and a subspace
complement must be preserved in any further search.

\begin{question}
Which 64-point spectral tiles in $\F^9$ have no subspace complement,
up to affine changes of coordinates, and how varied can their complete
families of spectra be?
\end{question}
Appendix~\ref{app:comparison} gives two affinely inequivalent domains
with exactly the same families of spectra and complements. Thus even
these two complete families together do not determine the affine class
of the domain. It remains to determine how many affine classes share
these families, and what other families can occur.
An accompanying question is which values of $\wsp$ occur for those
domains. Here the complete classification forces $15/8$; a list of some
spectra for a different domain would give only an upper bound on its
minimum.

\clearpage
\appendix
\section{A common class of constructions and comparison with Goyal}
\label{app:comparison}

\emph{Comparison added 30 October 2026.} The version of Goyal's project
used here is specified in the bibliography.

Goyal~\cite{Goyal} gives another 64-point spectral tile with no periodic
complement, and classifies its tiling complements. His construction uses
the same eight-point set $C$ as ours, but different layer subspaces.
We first isolate the part of our argument that applies to both
constructions. We then give complete finite data for his example and
prove that the two domains are not related by an affine change of
coordinates. All coordinates and Fourier normalizations remain those
of the main text.

\subsection{What the common difference set determines}

Recall $H=\F^6$, $Y=\F^3$, $G=H\times Y$, $S=C+C$, and
$Z=H\setminus S$. The next theorem shows that, within this class of
layer constructions, the difference set determines both entire
families of complements and spectra. The essential additional fact
for spectra is Lemma~\ref{lem:J}, which recovers the union of the
annihilators from the union of the layers.

\begin{theorem}[A common layer construction]\label{thm:common-layers}
Let $U_y\le H$ be a three-dimensional linear subspace for each $y\in Y$,
and put
\[
 A=\bigcup_{y\in Y}(U_y\times\{y\}).
\]
Suppose that
\begin{equation}\label{eq:common-missing}
 G\setminus(A+A)=(S\setminus\{0\})\times\{0\}.
\end{equation}
Then $A$ has 64 points, tiles $G$, and is spectral. Its tiling complements
are exactly
\[
 (C\times\{0\})+g\qquad(g\in G),
\]
and its spectra are exactly
\[
 (C+a)\times Y\qquad(a\in H).
\]
There are 512 distinct complements and 64 distinct spectra, with eight
of each containing zero and one translation class of each. No complement
is periodic or an affine subspace. Every spectrum has Wiener norm
$15/8$, so $\wsp(A)=15/8$.
\end{theorem}

\begin{proof}
The eight disjoint layers each have eight points. A sum of two points
of $A$ has vertical coordinate zero precisely when the points are in
the same layer, whose horizontal differences form $U_y+U_y=U_y$.
Taking that horizontal part of~\eqref{eq:common-missing} gives
\[
 \bigcup_y U_y=Z\cup\{0\}.
\]
Every nonzero vector of every $U_y$ therefore has weight three or four.
Lemma~\ref{lem:J}, and the fact that $J(Z)=Z$, give
\begin{equation}\label{eq:common-dual}
 \bigcup_y U_y^\perp=J\left(\bigcup_y U_y\right)=Z\cup\{0\}.
\end{equation}

Write $T_0=C\times\{0\}$. Its difference set is $S\times\{0\}$,
which meets $A+A$ only at zero by~\eqref{eq:common-missing}.
Since $|A||T_0|=64\cdot8=512$, Lemma~\ref{lem:tiling} proves
$A\oplus T_0=G$. For any other complement $T$, translate $T$ to contain
zero. Its nonzero points and all its nonzero differences must then lie
in $(S\setminus\{0\})\times\{0\}$. Thus $T=Q\times\{0\}$ with
$|Q|=8$ and $Q+Q\subseteq S$. Lemma~\ref{lem:Q} makes $Q$ a translate
of $C$. Undoing the normalization gives the stated list, and translating
$T_0$ shows that every listed set is indeed a complement.

For spectrality, let $\Lambda_0=C\times Y$. Equation~\eqref{eq:productW}
shows that $W_{\Lambda_0}$ is supported on $S\times\{0\}$. By
\eqref{eq:common-missing} it vanishes on $(A+A)\setminus\{0\}$.
The matrix criterion of Lemma~\ref{lem:spectral}, using row orthogonality
and $|\Lambda_0|=|A|=64$, proves that $\Lambda_0$ is a spectrum.

Conversely, let $\Gamma$ be any spectrum of $A$. Row orthogonality gives
$W_\Gamma(d)=0$ for every $d\in(A+A)\setminus\{0\}$, so
\[
 \supp W_\Gamma\subseteq S\times\{0\}.
\]
By Lemma~\ref{lem:period}, its indicator is invariant under every
vertical translation. Hence $\Gamma=Q\times Y$ for some eight-point
set $Q\subset H$. For distinct $q,q'\in Q$, put $u=q+q'$. Each
$(u,v)$, $v\in Y$, is a nonzero difference of two frequencies in
$\Gamma$. Column orthogonality therefore gives
\[
 0=W_A(u,v)=\sum_{y\in Y}(-1)^{v\cdot y}W_{U_y}(u)
 \qquad(v\in Y).
\]
Fourier inversion on $Y$ makes $W_{U_y}(u)=0$ for every $y$.
The transform of a subspace indicator, Lemma~\ref{lem:subtransform},
then gives $u\notin U_y^\perp$ for every $y$. By
\eqref{eq:common-dual}, $u\in S\setminus\{0\}$. Consequently
$Q+Q\subseteq S$, and Lemma~\ref{lem:Q} gives $Q=C+a$.
All these sets are spectra, since they are translates of $\Lambda_0$.

Proposition~\ref{prop:allT} proves that $C$ has no nonzero period and
is not an affine subspace. Those facts concern $C$ alone. They give
512 distinct translates of $T_0$, none periodic or affine, and 64
distinct translates of $C$. The respective translates contain zero
exactly when their translation parameters belong to $T_0$ or $C$,
giving eight of each. Finally, every listed spectrum is a translate of
$\Lambda_0$, whose norm was computed in Corollary~\ref{cor:norm}:
\[
 \W(\Lambda_0)=\frac{64+28\cdot32}{512}=\frac{15}{8}.
\]
Since the list is exhaustive, this is also the minimum over all spectra.
\end{proof}

\subsection{Goyal's domain and explicit certificates}

To distinguish the two domains, write $E_{\mathrm G}$ for Goyal's set.
Use the integer labels of Section~\ref{sec:construction}. Define
\begin{equation}\label{eq:goyal-domain}
 E_{\mathrm G}=\bigcup_{y\in Y}(K_y\times\{y\}),
\end{equation}
where Table~\ref{tab:goyal-spaces} specifies $K_y$ completely.
The first three entries in each row are its basis $a,b,c$; the remaining
entries are their other nonempty sums. This is the construction in
\cite[Table~1]{Goyal}, with its sums printed here to make the following
checks independent of software.

\begin{table}[htbp]
\centering
\begin{tabular}{c|rrr|rrr|r}
\toprule
$y$&$a$&$b$&$c$&$a+b$&$a+c$&$b+c$&$a+b+c$\\\midrule
0&11&22&37&29&46&51&56\\
1&7&26&43&29&44&49&54\\
2&14&23&35&25&45&52&58\\
3&13&22&39&27&42&49&60\\
4&7&26&43&29&44&49&54\\
5&11&21&39&30&44&50&57\\
6&15&19&38&28&41&53&58\\
7&11&23&38&28&45&49&58\\\bottomrule
\end{tabular}
\caption{Goyal's layer subspaces. Add zero to a row to obtain $K_y$.
Only rows 1 and 4 coincide.}\label{tab:goyal-spaces}
\end{table}

\begin{proposition}\label{prop:goyal}
The set $E_{\mathrm G}$ satisfies~\eqref{eq:common-missing}. Thus
Theorem~\ref{thm:common-layers} classifies all its tiling complements
and all its spectra, and gives $\wsp(E_{\mathrm G})=15/8$.
\end{proposition}

\begin{proof}
Each row has seven distinct nonzero entries, so its three basis vectors
are independent. Every entry has weight three or four. The union of the
rows consists of the following vectors, separated by weight:
\[
\begin{array}{c|l}
3&7,11,13,14,19,21,22,25,26,28,35,37,38,41,42,44,49,50,52,56\\
4&15,23,27,29,30,39,43,45,46,51,53,54,57,58,60.
\end{array}
\]
These are all $\binom63=20$ vectors of weight three and all
$\binom64=15$ vectors of weight four. Hence
$\bigcup_yK_y=Z\cup\{0\}$, which gives the horizontal part of
$E_{\mathrm G}+E_{\mathrm G}$.

For a nonzero vertical difference $z$, that difference set contains
$(K_i+K_j)\times\{z\}$ whenever $i+j=z$. The necessary pairs and
intersections are displayed in Table~\ref{tab:goyal-cross}.
The intersection column is obtained by comparing the complete rows
in Table~\ref{tab:goyal-spaces}.

\begin{table}[htbp]
\centering
\begin{tabular}{cccc}
\toprule
$z$&$(i,j)$&$K_i\cap K_j$&$K_i+K_j$\\\midrule
1&(2,3)&$\{0\}$&$H$\\
2&(0,2)&$\{0\}$&$H$\\
3&(1,2)&$\{0\}$&$H$\\
5&(3,6)&$\{0\}$&$H$\\
6&(0,6)&$\{0\}$&$H$\\
7&(1,6)&$\{0\}$&$H$\\\midrule
4&(0,4)&$\{0,29\}$&$29^\perp$\\
4&(1,5)&$\{0,44\}$&$19^\perp$\\
4&(3,7)&$\{0,49\}$&$14^\perp$\\\bottomrule
\end{tabular}
\caption{Cross-layer certificates for $E_{\mathrm G}$. Here
$n^\perp=\{x\in H:n\cdot x=0\}$.}\label{tab:goyal-cross}
\end{table}

A zero intersection in the first six rows makes all $8\cdot8=64$
sums distinct, so the sum is $H$. For each of the last three rows,
the intersection has dimension one, and the sum has dimension five.
For completeness, the dimension formula used here follows by applying
rank--nullity to $(u,v)\mapsto u+v$ from $K_i\times K_j$: its kernel
is $\{(w,w):w\in K_i\cap K_j\}$. The indicated normal vector $n$
is nonzero and has dot product zero with each of the six displayed
basis vectors of $K_i$ and $K_j$. Indeed, each of the six basis vectors has exactly two coordinates
equal to one in common with $n$, so the dot product is zero. Thus $K_i+K_j\subseteq n^\perp$;
both sides have dimension five, proving equality.

The three hyperplanes in the last rows cover $H$. Indeed, $29+19=14$
in the vector labeling. If $x$ belongs to neither $29^\perp$ nor
$19^\perp$, then both pairings equal one, and
\[
 14\cdot x=(29+19)\cdot x=1+1=0.
\]
Thus $x\in14^\perp$. All seven nonzero vertical differences consequently
have the full horizontal space $H$. Together with the horizontal union
already proved, this gives
\[
 E_{\mathrm G}+E_{\mathrm G}
 =((Z\cup\{0\})\times\{0\})
   \cup(H\times(Y\setminus\{0\})),
\]
which is equivalent to~\eqref{eq:common-missing}.
\end{proof}

Goyal's manuscript already supplies the explicit spectrum, the tiling,
the obstruction to periodic complements, and the complete classification
of tiling complements. The conclusion about \emph{every} spectrum and
its Wiener norm follows here from Theorem~\ref{thm:common-layers}.
The cited project also supplies a Lean~4 formalization of its spectral
tile and its complement classification. That is a separate formal
verification of those statements, not a formalization of the spectrum
classification proved here or of the dimension-eight positive theorem.
Our proofs above use the displayed finite identities and do not depend
on replaying that formalization.

\subsection{An affine invariant that distinguishes the domains}

The layer directions look different in the two tables. To turn that
observation into a proof of inequivalence, we must show that the relevant
layer space can be recovered from the domain itself. Otherwise a change
of coordinates could replace the chosen layers with different ones.

\begin{lemma}\label{lem:intrinsic-layer-space}
For $A\subseteq G$, put
\[
 V_A=\spanof\bigl(G\setminus(A+A)\bigr).
\]
If $A'=M(A)+b$ for an invertible linear map $M:G\to G$, then
$V_{A'}=M(V_A)$. Suppose that every nonempty intersection of $A$ with
a coset of $V_A$ is an affine subspace. The multiplicities of their
distinct direction subspaces are then unchanged by affine equivalence.
\end{lemma}

The direction of an affine subspace $a+U$ means $U$; it can also be
recovered as $(a+U)+(a+U)=U$. Its multiplicity is the number of the
coset intersections having that same direction. For example, a list
$U,U,V$ with $U\ne V$ has multiplicities $(2,1)$.

\begin{proof}
In characteristic two the two occurrences of $b$ cancel, giving
$A'+A'=M(A+A)$. Since $M$ is a bijection of $G$, it carries the
complement of $A+A$ to the complement of $A'+A'$. Taking spans proves
the first assertion. The affine map therefore bijects the cosets of
$V_A$ with those of $V_{A'}$. A nonempty coset intersection with direction
$U$ is carried to one with direction $M(U)$. Since $M$ is injective,
$M(U)=M(U')$ exactly when $U=U'$. The numbers of occurrences of the
distinct directions are consequently preserved.
\end{proof}

\begin{proposition}\label{prop:inequivalence}
The domains $E$ and $E_{\mathrm G}$ are not affinely equivalent,
although their families of tiling complements and spectra are identical
in the displayed coordinates.
\end{proposition}
\begin{proof}
Both domains have missing differences
$(S\setminus\{0\})\times\{0\}$. This set contains
$(e_1,0),\ldots,(e_6,0)$ and is contained in $H\times\{0\}$.
Its span is therefore exactly $H\times\{0\}$ for both domains.
The eight coset intersections are their eight layers, with directions
$B_y\times\{0\}$ or $K_y\times\{0\}$.

In~\eqref{eq:layers}, $A_0$ and $D_0$ each occur twice, and the four
other distinct subspaces each occur once. Thus $E$ has direction
multiplicities $(2,2,1,1,1,1)$. In Table~\ref{tab:goyal-spaces}, only
$K_1=K_4$ is repeated, giving $(2,1,1,1,1,1,1)$. By
Lemma~\ref{lem:intrinsic-layer-space}, the different multiplicities
exclude every affine equivalence. The equality of the complement and
spectrum families follows from Theorem~\ref{thm:common-layers}, applied
to both domains.
\end{proof}

Thus knowledge of every tiling complement and every spectrum does not,
even in this explicit setting, determine the domain up to affine
equivalence. This comparison concerns the two displayed domains and
does not assert a classification of all domains with these families.

\end{document}